\documentclass[11pt]{article}
\usepackage{amsmath,amssymb,amsthm}
\usepackage{epsfig}
\usepackage{color}
\usepackage{enumerate}
\usepackage{algorithm}
\usepackage{algorithmic}
\usepackage{hyperref}
\usepackage{enumerate}
\usepackage{graphicx}
\usepackage{tikz}
\usepackage{tkz-graph}
\usepackage{tkz-berge}
\usetikzlibrary{arrows.meta,shapes.arrows}
\hypersetup{colorlinks=true, linkcolor=blue, citecolor=blue,
urlcolor=blue}

\hypersetup{colorlinks=true}

\hypersetup{colorlinks=true, linkcolor=blue, citecolor=blue,urlcolor=blue}

\title{On $k$-(total) limited packing in graphs}
\date{}
\author {
Azam Sadat Ahmadi\ \ and Nasrin Soltankhah\thanks{Corresponding author}\vspace{2mm}\\
Department of Mathematics, Faculty of Mathematical Sciences, Alzahra University,\\ Tehran, Iran\vspace{1mm}\\
{\tt as.ahmadi@alzahra.ac.ir}\\
{\tt soltan@alzahra.ac.ir}\\
}
\date{}

\newtheorem{theorem}{Theorem}[section]
\newtheorem{corollary}[theorem]{Corollary}
\newtheorem{lemma}[theorem]{Lemma}

\newtheorem{proposition}[theorem]{Proposition}

\theoremstyle{definition}

\theoremstyle{remark}

\begin{document}

\maketitle

\begin{abstract}
A set $B\subseteq V(G)$ is called a $k$-total limited packing set in a graph $G$ if $|B\cap N(v)|\leq k$ for any vertex $v\in V(G)$. The $k$-total limited packing number $L_{k,t}(G)$ is the maximum cardinality of a $k$-total limited packing set in $G$. Here, we give some results on the $k$-total limited packing number of graphs  emphasizing trees, especially when $k=2$. We also study the $2$-(total) limited packing number of some product graphs. 

A $k$-limited packing partition ($k$LPP) of graph $G$ is a partition of $V(G)$ into $k$-limited packing sets. The minimum cardinality of a $k$LPP is called the $k$LPP number of $G$ and is denoted by $\chi_{\times k}(G)$, and we obtain some results for this parameter.
\end{abstract}
\vspace{0.5mm}

\section{Introduction and preliminaries} 

In this work, we consider $G=\big(V(G),E(G)\big)$ as a finite simple graph.
$N_G(v)$ and $N_G[v]=N_G(v)\cup \{v\}$ are used to refer to the {\em open neighborhood} and {\em closed neighborhood} of a vertex $v\in V(G)$, respectively.
The {\em minimum} and {\em maximum degrees} of a graph $G$ are denoted by $\delta(G)$ and $\Delta(G)$, respectively.
We refer to \cite{West} as a source for terminology and notation that is not explicitly defined here.

$G^{-}$ indicates the graph obtained from $G$ by removing its isolated vertices. By $G[S]$, we mean the subgraph induced by the subset $S$ of vertices in $G$.

A set of vertices $S\subseteq V(G)$ is called a \textit{dominating set} (DS) in $G$ if every vertex not in $S$ is adjacent to at least one vertex in $S$. The \textit{domination number} of $G$, denoted $\gamma(G)$, is the smallest number of vertices in a dominating set of $G$.
A set $S\subseteq V(G)$ is a \textit{total dominating set} (TDS) in the graph $G$ if every vertex in $V(G)$ is adjacent to a vertex of $S$. The \textit{total domination number} of $G$, denoted $\gamma_{t}(G)$, is the smallest number of vertices in a total dominating set of $G$.

A set of vertices $S\subseteq V(G)$ with $\delta(G)\geq k-1$ is said to be a \textit{$k$-tuple dominating set} ($k$TD set) in $G$ provided that for every $v\in V(G)$, we have $|N[v]\cap S|\geq k$ . The \textit{$k$-tuple domination number} $\gamma_{\times k}(G)$ of graph $G$ is the number of vertices in a smallest $k$TD set in $G$. A \textit{$k$-tuple domatic partition} ($k$TD partition) of a graph $G$ is a partition of the vertices of $G$ into $k$TD sets. The largest number of sets that can be obtained from a vertex partition of $G$ into $k$TD sets is called the \textit{$k$-tuple domatic number} and is denoted by $d_{\times k}(G)$. Notice that when $k=1$, $S$ and $\gamma_{\times1}(G)$ are the usual dominating set and domination number $\gamma(G)$, respectively. Additionally, $d_{\times1}(G)=d(G)$ refers to the well-studied domatic number (see \cite{ch}).

A vertex subset $B$ of a graph $G$ is called a \textit{packing} (resp. an \textit{open packing}) provided that $|B\cap N[v]|\leq 1$ (resp. $|B\cap N(v)|\leq 1$) for each  vertex $v\in V(G)$. The \textit{packing number} $\rho(G)$ and \textit{open packing number} $\rho_{o}(G)$ are defined as the maximum cardinality of a packing set and an open packing set, respectively. To obtain additional information on these concepts, the reader can refer to \cite{hhh} and \cite{hhs}.

In $2010$, the concept of limited packing (LP) in graphs was introduced by Gallant et al. \cite{gghr}. A \textit{$k$-limited packing} ($k$LP) in a graph $G$ is a set $B\subseteq V(G)$ such that for each vertex $v$ of $V(G)$, the cardinality of the intersection of $B$ and $N[v]$ is at most $k$. The maximum cardinality of a $k$-limited packing set in G is called the \textit{$k$-limited packing number} $L_{k}(G)$. They also presented some real-world applications of this concept in network security, market situation, NIMBY and codes. This topic was next investigated in numerous papers, such as references \cite{bcl}, \cite{gz} and \cite{s}. Similarly, a \textit{$k$-total limited packing} ($k$TLP) in $G$ is a set $B\subseteq V(G)$ such that for each vertex $v$ of $V(G)$, the cardinality of the intersection of $B$ and $N(v)$ is at most $k$. The maximum cardinality of a $k$-total limited packing set in G is called the \textit{$k$-total limited packing number} $L_{k,t}(G)$. This topic was initially studied in \cite{hms}, and some theoretical applications of it were given in \cite{ass, hmsv}. It is worth noting that the latter two concepts are identical to packing and open packing when $k$ equals $1$. Notice that a $k$LP set is a $k$TLP set, too.

For the \textit{cartesian product} of graphs $G$ and $H$, denoted $G\square H$, and the \textit{direct product} of graphs $G$ and $H$, denoted $G\times H$, the vertex set of the product is $V(G)\times V(H)$. Their edge sets are defined as follows. In $G\square H$, two vertices are adjacent if they are adjacent in one coordinate and equal in the other. In $G\times H$ two vertices are adjacent if they are adjacent in both coordinates.

Suppose that $G$ is a labeled graph on $n$ vertices, and $\mathcal{H}$ is a sequence of $n$ rooted graphs $H_1, H_2, \cdots, H_n$. If we identify the $i^{th}$ vertex of $G$ with the root of $H_i$, we obtain a new graph called the \textit{rooted product} graph. This graph is denoted by $G(\mathcal{H})$. We here focus on the special case of rooted product graphs for which $\mathcal{H}$ consists of $n$ isomorphic rooted graph. Assume that $v$ is the root vertex of $H$, we define the rooted product graph $G\circ_{v} H=(V,E)$, such that $V= V(G)\times V(H)$ and $$E=\displaystyle \bigcup_{i=1}^{n} \big\{(g_i,h)(g_i,h'):hh'\in E(H)\big\}\cup \big\{(g_i,v)(g_j,v):g_ig_j\in E(G)\big\}.$$

For $g\in V(G)$, $h\in V(H)$ and $\ast \in \{\square, \times, \circ_{v}\}$, we call $G^h=\{(g,h)\in V(G\ast H)|g\in V(G)\}$ a $G$-layer through $h$, and ${}^g\!H=\{(g,h)\in V(G\ast H)|h\in V(H)\}$ an $H$-layer through $g$ in $G\ast H$.

Notice that the subgraphs induced by the $H$-layers (resp. the $G$-layers) of $G\circ_{v} H$ (or $G\square H$) are isomorphic to $H$ (resp. to $G$). However, there are no edges between the vertices of $G^h$ and the vertices of ${}^g\!H$ in direct product $G\times H$.

 A $k$-\textit{limited packing partition} ($k$LPP) of a graph $G$ is a partition of the vertices of $G$ into $k$LP sets. The smallest number of sets that can be obtained from a vertex partition of $G$ into $k$LP sets is called the \textit{$k$-limited packing partition number} ($k$LPP number) and is denoted by $\chi_{\times k}(G)$. This concept can also be considered as the dual of $k$TD partition problem. Our main focus for $k$TLP sets is on $k=2$. This is because for larger values of $k$, we lose some significant families of graphs (for instance, $\gamma_{\times k}$ and $d_{\times k}$ cannot be defined for trees when $k\geq3$) or we encounter trivial problems (for instance, $L_{\times k}(G)=|V(G)|$ and $\chi_{\times k}(G)=1$ if $k\geq \Delta(G)+1$). On the other side, many results for $k\in \{1,2\}$ may be generalized to the general case $k$. Inaddition, Stronger results may be obtained for small values of $k$. 

Here, we first discuss $k$TLP, especially when $k=2$, and give several sharp bounds for it. Then, we improve some of these inequalities for trees. In Section $3$, we bound $L_{2}$ and $L_{2,t}$ for the cartesian product, direct product and rooted product graphs .
In Section $4$, we give a lower bound for $\chi_{\times k}$, and determine the values of $\chi_{\times 2}$ for the corona product.
For the sake of convenience, for any graph $G$ by an $\eta(G)$-set with $\eta\in \{L_{k},\gamma_{t},\rho,\rho_{o},L_{k,t}\}$ we mean a $k$LP set, TD set, packing set, open packing set and $k$TLP set in $G$ of cardinality $\eta(G)$, respectively.


\section{Results on $k$-total limited packing}
If $G$ is a graph of order $n$ and $k\geq n-1$, then $L_{k,t}(G)=n$.
Note that $k\geq \Delta(G)$ is a weaker condition than the previous one. Therefore, we only need to compute the $k$TLP number for those graphs G such that $k<\Delta(G)$.

If $B\subseteq V(G)$ and $|B|=k$, then $|B\cap N(v)|\leq k$ for each vertex $v$ of $V(G)$. So, $k\leq L_{k,t}(G)\leq n$.\\
We give some upper bounds for the $k$TLP number of a graph in the following.

\begin{theorem}
Let $G$ be a graph of order $n\geq 2$ with degree sequence $d_1, d_2, \cdots, d_n$ such that $d_1\leq d_2\leq \cdots \leq d_n$. Then
$$L_{k,t} (G)\leq max~\{t|d_1+d_2+\cdots+d_t\leq kn\},$$
and this bound is sharp.
\end{theorem}
\begin{proof}
Let $\mathcal{B}=\{v_1, v_2, \cdots, v_{|\mathcal{B}|}\}$ be an $L_{k,t}(G)$-set. Then
$$d_1+d_2+\cdots+d_{|\mathcal{B}|}\leq deg(v_1)+ deg(v_2)+ \cdots+ deg(v_{|\mathcal{B}|})\leq k|\mathcal{B}|+k(n-|\mathcal{B}|).$$
So $d_1+d_2+\cdots+d_{|\mathcal{B}|}\leq kn$. Therefore, $L_{k,t}(G) \in \{t|d_1+d_2+\cdots+d_t\leq kn\}$.

The sharpness of this bound can be seen as follows. Suppose that $G$ is a complete graph of order at least  $k+2$. Then, it is easy to see that $L_{k,t}(G)=k$. On the other hand, $k=L_{k,t}(G)\leq max~\{t|t(n-1)\leq kn\}=k$. 
\end{proof}

\begin{lemma}\label{l9}
If $G$ is a graph of order $n$, then $L_{k,t}(G)\leq n+k-\Delta(G)$.
\end{lemma}
\begin{proof}
Assume that $w$ is a vertex of maximum degree in $G$. If $k\geq \Delta(G)$, then it is clear that $V(G)$ is a $k$TLP set of $G$. So, $L_{k,t}(G)=n\leq n+k-\Delta(G)$. Thus, we assume that $k< \Delta(G)$. Let $S$ be an $L_{k,t}(G)$-set. Since $|N(w)\cap S|\leq k$, there is at least $\Delta(G)-k$ vertices in $N(w)\backslash S$. Hence, $|\overline{S}|\geq \Delta(G)-k$. Therefore, we have $L_{k,t}(G)=|S|=n-|\overline{S}|\leq n+k-\Delta(G)$.
\end{proof}

We define the family $\Omega$ consisting of all graphs $G$ constructed as follows.\\
 Suppose that $G$ is a graph of order $n$ such that $V(G)=A\cup B$ has the following conditions:
\begin{itemize}
\item[(i)]
$|A\cap B|=3$,
\item[(ii)]
$G[A]$ has a spanning star, and each component of $G[B]$ is a path or a cycle,
\item[(iii)]
for every vertex $v\in \overline{B}$, we have $|N(v)\cap B|\leq 2$ .
\end{itemize}

Figure \ref{G} depicts a representative member of $\Omega$.

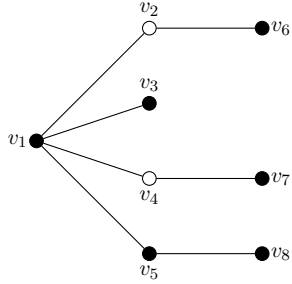
\begin{figure}[h]
\centering
\begin{tikzpicture}[scale=.5, transform shape]
\node [draw, shape=circle,fill=black] (v1) at (-1,0) {};
\node [draw, shape=circle] (v2) at (2,3) {};
\node [draw, shape=circle,fill=black] (v6) at (5,3) {};
\node [draw, shape=circle,fill=black] (v3) at (2,1) {};
\node [draw, shape=circle] (v4) at (2,-1) {};
\node [draw, shape=circle,fill=black] (v7) at (5,-1) {};
\node [draw, shape=circle,fill=black] (v5) at (2,-3) {};
\node [draw, shape=circle,fill=black] (v8) at (5,-3) {};

\node [scale=1.5] at (-1.5,0) {$v_1$};
\node [scale=1.5] at (2,3.5) {$v_2$};
\node [scale=1.5] at (2,1.5) {$v_3$};
\node [scale=1.5] at (2,-1.5) {$v_4$};
\node [scale=1.5] at (2,-3.5) {$v_5$};
\node [scale=1.5] at (5.5,3) {$v_6$};
\node [scale=1.5] at (5.5,-1) {$v_7$};
\node [scale=1.5] at (5.5,-3) {$v_8$};

\draw(v1)--(v2)--(v6);
\draw(v1)--(v3);
\draw(v1)--(v4)--(v7);
\draw(v1)--(v5)--(v8);
\end{tikzpicture}
\caption{A graph $H\in \Omega$ with $A=\{v_1,v_2,v_3,v_4,v_5\}$ and $B=\{v_1,v_3,v_5,v_6,v_7,v_8\}$.}\label{G}
\end{figure} 

The next theorem shows that $\Omega$ is the set of all graphs $G$ of order $n$ satisfying $L_{2,t}(G)= n+2-\Delta(G)$.

\begin{theorem}\label{t10}
If $G$ is a graph of order $n$, then $L_{2,t}(G)\leq n+2-\Delta(G)$.
Furthermore, $L_{2,t}(G)= n+2-\Delta(G)$ if and only if $G\in \Omega$.
\end{theorem}
\begin{proof}
Suppose that $S$ is an $L_{2,t}(G)$-set, and $w$ is a vertex of maximum degree in $G$. Notice that each component of $G[S]$ is a path or a cycle, and we have $L_{2,t}(G)=|S|=n-|\overline{S}|\leq n+2-\Delta(G)$ by Lemma \ref{l9}.

If $L_{2,t}(G)=n+2-\Delta(G)$, then $|\overline{S}|= \Delta(G)-2$, $\big(V(G)\backslash N[w]\big)\subseteq S$ and $|N[w]\cap S|=3$.
Based on the above argument, we have $G\in \Omega$ with $N[w]=A$ and $S=B$.

Let now $G\in \Omega$, then it suffices to prove that $L_{2,t}(G)\geq n+2-\Delta(G)$. Assume now that $A\cap B=\{u_1,u_2,u_3\}$ and $|A|=a+1$, where $w$ is a vertex of degree $a$ in $G[A]$. We claim that $\Delta(G)=a$. Each vertex $v \in B$ in $G[B]$ is at most of degree two. So, each of the vertices $u_1, u_2$ and $u_3$ is adjacent to at most two vertices in $B$. On the other hand, each of $u_1, u_2, u_3$ is adjacent to at most $a-2$ vertices in $A\backslash \{u_1, u_2, u_3\}$. Thus, deg$(u_1)\leq a$, deg$(u_2)\leq a$ and deg$(u_3)\leq a$. For each vertex $v\in A\backslash \{u_1, u_2, u_3\}$, $v$ is adjacent to at most $a-3$ vertices in $A\backslash \{u_1, u_2, u_3,v\}$ and to at most two vertices in $B$. So deg$(v)\leq a-1$ for every $v\in A\backslash \{u_1, u_2, u_3\}$. For each vertex $v'\in B\backslash \{u_1, u_2, u_3\}$, $v'$ is adjacent to at most $a-2$ vertices in $A\backslash \{u_1, u_2, u_3\}$ and to at most two vertices in $B$. Thus, deg$(v')\leq a$ for every $v'\in B\backslash \{u_1, u_2, u_3\}$. Thus, $\Delta (G)\leq a$. But deg$(w)\geq a$, which implies that $\Delta (G)=a$. Note that $B$ is a $2$TLP set of $G$ with $|B|=n-|A|+3=n+2-\Delta (G)$. Therefore, we have $L_{2,t}(G)\geq n+2-\Delta(G)$.
\end{proof}

\begin{corollary}\label{c11}
Let $G$ be an $r$-regular graph of order $n$ such that $L_{k,t}(G)=n+k-r$ for $k\leq r-1$. Then, we have $r\geq \frac{n+1}{2}$.
\end{corollary}
\begin{proof}
If $r=n-1$, then $G$ is a complete graph with $L_{k,t}(G)=k$ for $1\leq k\leq n-2$. So, let $r\leq n-2$. Now assume that $w\in V(G)$ and that $S$ is an $L_{k,t}(G)$-set with $|S|= n+k-r$. Since $|N(w)\cap S|\leq k$, it follows that $|N(w)\cap \overline{S}|\geq r-k$ . Obviously, $|\overline S|=n-|S|=r-k$. Thus, there are exactly $r-k$ vertices, namely $v_1, v_2, \cdots, v_{r-k}$, in $N(w)\cap \overline{S}$. Moreover, $\overline{S}=\{v_1, v_2, \cdots, v_{r-k}\}$. Let $U=V(G)\backslash N(w)$, clearly $U\subseteq S$, and $U\neq \emptyset$ since $r\leq n-2$. If $u\in U$, then $|N(u)\cap S|\leq k$. So, any vertex $u\in U$ is adjacent to all vertices in $\overline{S}$, i.e. every vertex $v_i\in \overline{S}$ is adjacent to all $n-r$ vertices in $U$. Note that $v_i$ has at least one neighbor in $N[w]$, and deg$(v_i)=r$. Therefore, $n-r+1\leq r$ and we get $r\geq \frac{n+1}{2}$.
\end{proof}

For any tree $T$ with order $n\geq3$, $\delta'(T)$ denotes the minimum degree of $T$ among all non-leaf vertices.

\begin{theorem}
Let $T$ be a tree of order $n\geq3$ for which $\delta'(T)\geq c$, and let $c\geq 4$ be a positive integer. Then, we have  $L_{2,t}(T)\leq \frac{c-2}{c-1} n-c+4$.
\end{theorem}
\begin{proof}
We prove this theorem by induction on the order of tree $T$. Since $\delta'(T)\geq c$, we have $n\geq c+1$. If $n\in \{c+1, c+2, \cdots, 2c-1\}$, then $T\in \{K_{1,c}, K_{1,c+1}, \cdots, K_{1,2c-2}\}$, respectively. Hence, $L_{2,t}(T)=3\leq \frac{c-2}{c-1}n-c+4$.
Assume that for all tree $T'$ of order $n'<n$ with $\delta'(T')\geq c$, we have $L_{2,t}(T')\leq \frac{c-2}{c-1}n'-c+4$. Now let $T$ be a tree of order $n\geq 2c$ such that $\delta'(T)\geq c$ and let  $S$ be an $L_{2,t}(T)$-set. We root $T$ at $r$, and suppose $v'$ is a leaf of $T$ at the furthest distance from $r$, and $v''$ is the parent of $v'$. Assume that $L$ is the set of all leaves in $N(v'')$. Since $v''$ is adjacent to at least $c-1$ leaves, it follows that $|L|\geq c-1$. Suppose that $T''$ be obtained from $T$ by deleting all the vertices of $L$.
By the induction hypothesis, we have $L_{2,t}(T'')\leq \frac{c-2}{c-1}|V(T'')|-c+4\leq \frac{c-2}{c-1}(n-(c-1))-c+4=\frac{c-2}{c-1}n-2c+6$.

On the other hand, $|L\cap S|\leq |N(v'')\cap S|\leq 2$. Therefore, we get $L_{2,t}(T)\leq L_{2,t}(T'')+2\leq \frac{c-2}{c-1}n-2c+8 \leq \frac{c-2}{c-1}n-c+4$.
\end{proof}

\begin{theorem}
If $G$ is a graph, then for any edge $e\in E(G)$,
$$L_{k,t}(G)\leq L_{k,t}(G-e)\leq L_{k,t}(G)+2.$$
Furthermore, these bounds are sharp.
\end{theorem}
\begin{proof}
Any $k$TLP set of $G$ is also a $k$TLP set of $G-e$, so $L_{k,t}(G)\leq L_{k,t}(G-e)$. Moreover, if $C$ is a cycle on $n$ vertices, then $L_{2,t}(C)= L_{2,t}(C-e)$ for every edge $e\in E(C)$.

Suppose now that $B$ is an $L_{k,t}(G-e)$-set and $e=uv$. If $u,v\in B$, then $B-\{u,v\}$ is a $k$TLP set of $G$ ,and hence $L_{k,t}(G)\geq|B|-2$. If $u\in B$ and $v\notin B$, then $B-\{u\}$ is a $k$TLP set of $G$, and $L_{k,t}(G)\geq|B|-1$. If $u,v\notin B$, then $B$ is a $k$TLP set of $G$, and we have $L_{k,t}(G)\geq|B|$. Therefore, $L_{k,t}(G-e)\leq L_{k,t}(G)+2$.

Let $G$ be a double star $ST(x,y)$, which is the graph obtained by joining the centers of two stars $K_{1,x}$ and $K_{1,y}$ with an edge, such that $x,y\geq k+1$. Assume that the center of stars are $u$ and $v$, respectively. Then, $L_{k,t}(G-e)= L_{k,t}(G)+2$ for $e=uv$.
\end{proof}

\begin{proposition}
Let $G$ be a graph without isolated vertices such that $\Delta(G)\geq 2$, then
$$\rho_o(G)+1 \leq L_{2,t}(G) \leq \frac{{\Delta(G)}^2+1}{\delta(G)}\rho_o(G).$$
\end{proposition}
\begin{proof}
Mojdeh et al. \cite{hms} showed that the lower bound is true for $\Delta(G)\geq 2$. So, we just verify the upper bound. Let $v\in V(G)$ be an arbitrary vertex, then the set of all vertices at  distance at most two from $v$ has at most ${\Delta(G)}^2+1$ vertices. Thus, $\rho_o(G)\geq \frac{2n}{{\Delta(G)}^2+1}$, by the greedy algorithm.
Moreover, $L_{k,t}(G)\leq \frac{kn}{\delta(G)}$ \cite{hms}, and we get
$$\rho_o(G)\geq\frac{2n}{{\Delta(G)}^2+1}=\frac{2n\delta(G)}{{(\Delta(G)}^2+1)\delta(G)}\geq L_{2,t}(G)\frac{\delta(G)}{{\Delta(G)}^2+1}.$$
Therefore, we infer that
$$L_{2,t}(G)\leq \frac{{\Delta(G)}^2+1}{\delta(G)}\rho_o(G).$$
\end{proof}

We can improve the above bounds for trees as follows.

\begin{theorem}
If $T$ is a given tree with $\Delta(T)\geq2$, then
$$\rho_o(T)+1 \leq L_{2,t}(T) \leq 2\rho_o(T).$$
Moreover, the following statements hold:
\begin{itemize}
\item[(i)]
$\rho_o(T)+1 = L_{2,t}(T)$ if and only if T is a star with at least three vertices,
\item[(ii)]
$L_{2,t}(T)=2\rho_o(T)$ if and only if for every $L_{2,t}(T)$-set $S$ and every $\gamma_t(T)$-set $D$, we have $|N(s)\cap D|=1$ and $|N(d)\cap S|=2$ for any $s\in S$ and any $d\in D$.
\end{itemize}
\end{theorem}
\begin{proof}
The lower bound is true for $\Delta(T)\geq2$. We know that $L_{k,t}(T)\leq k\gamma_t(T)$ (\cite{hms}). On the other hand, we have $\rho_o(T)=\gamma_t(T)$ for every tree $T$ with at least two vertices (\cite{r}). As a consequence, we have $L_{2,t}(T) \leq2\gamma_t(T)= 2\rho_o(T)$.

Let $T$ be the star $K_{1,x}$ with $x\geq2$. Then, $\rho_o(T)=2$ and $L_{2,t}(T)=3$. Therefore, $\rho_o(T)+1 = L_{2,t}(T)$.

It remains for us to prove the converse. Assume now that $T$ is a tree with $\rho_o(T)+1 = L_{2,t}(T)$. We claim that $diam(T)\leq2$. Suppose to the contrary that there exist two vertices $v_1, v_4 \in V(T)$ such that $d(v_1,v_4)=3$ and let $P=v_1v_2v_3v_4$ be the path between them. Assume that $S_1$ is a $\rho_o(T)$-set, then $|V(P)\cap S_1|\leq2$. We now consider three cases as follows.

\textit{Case 1.} Let $V(P)\cap S_1=\emptyset$. Set $S_2=S_1\cup \{v_1,v_2\}$, we show that $S_2$ is a $2$TLP set of $T$. Since $|N(v_i)\cap S_1|\leq1$ and $|N(v_i)\cap \{v_1,v_2\}|\leq1$ for $1\leq i\leq 4$, it follows that $|N(v_i)\cap S_2|\leq2$ for every $v_i\in V(P)$. Let now $w$ be a vertex outside of $P$, so $|N(w)\cap\{v_1,v_2\}|\leq1$ because $T$ has no cycle. Thus, $|N(w)\cap S_2|\leq2$ for every vertex $w$ outside $P$. Therefore, we conclude that $S_2$ is a $2$TLP set of $T$.

\textit{Case 2.} Assume $V(P)\cap S_1=\{v_i\}$ for $1\leq i\leq 4$. First, let $i=1$ or $4$, by using similar techniques as in the previous case, $S_1\cup \{v_2,v_3\}$ is a $2$TLP set of $T$. If $i=2$ or $3$, then $S_1\cup \{v_1,v_4\}$ is a $2$TLP set of $T$.

\textit{Case 3.} Suppose $V(P)\cap S_1=\{v_i,v_j\}$ for some $1\leq i\neq j\leq4$. If $(i,j)\in \{(1,2), (1,4), (2,3), (3,4)\}$, then $S_1\cup \{v_3,v_4\}$,  $S_1\cup \{v_2,v_3\}$, $S_1\cup \{v_1,v_4\}$ and $S_1\cup \{v_1,v_2\}$ are $2$TLP sets of $T$, respectively.

In each case, we observe that $L_{2,t}(T)\geq \rho_o(T)+2$, which contradicts the assumption $\rho_o(T)+1 = L_{2,t}(T)$. Therefore, we deduce that $diam (T)\leq2$, and $T$ is a star with at least three vertices.

Let now $T$ be a tree with $\Delta(T)\geq2$. As mentioned earlier, we know that $L_{2,t}(T) =2 \rho_o(T)$ if and only if $L_{2,t}(T) =2\gamma_t(T)$. Let $S$ be an $L_{2,t}(T)$-set, and $D$ be a $\gamma_t(T)$-set. We now restate the proof of Theorem $7$ in \cite{hms}. We set $U=\{(s,d)\in V(G)\times V(G)|s\in S, d\in D~ \text{and}~ s\in N(d)\}$, and count the members of $U$ in two ways.
Since $|N(s)\cap D|\geq1$ for any $s\in S$, it follows that there is at least one vertex $d\in D$ such that $s\in N(d)$. Thus, $|S|\leq |U|$.
On the other hand, for any $d\in D$ we have $|N(d)\cap S|\leq2$. Hence, there exists at most two vertices $s_1, s_2\in S$ such that $s_1, s_2\in N(d)$. So, we get $|U|\leq 2|D|$, and $|S|\leq 2|D|$.
Therefore, $L_{2,t}(T) =2\gamma_t(T)$ if and only if the following statements hold:
\begin{itemize}
\item[(1)]
for any $s\in S$, we have $|N(s)\cap D|=1$,
\item[(ii)]
for any $d\in D$, we have $|N(d)\cap S|=2$.
\end{itemize}
\end{proof}

If $diam(G)=1$, then $G$ is a complete graph, and we know that $L_{2,t}(K_n)=2$. What can be said about the $2$TLP number of graphs with diameter $2$? The following theorem is the answer to this question.

\begin{theorem}
If $c\geq3$ is a positive integer, then there exists a graph $G$ with $diam(G)=2$ such that $L_{2,t}(G)=c$.
\end{theorem}
\begin{proof}
In what follows, we construct a graph $G$ with diameter $2$ for which $L_{2,t}(G)=c$. Suppose that $V_1=\{v_1, v_2, \cdots, v_c\}$ and $V_2=\{u_1, u_2, \cdots, u_{\frac{c(c-1)}{2}}\}$ with $V_1\cap V_2= \emptyset$. Let $G$ be a graph with vertex set $V(G)=V_1\cup V _2$ such that $G[V_1]=cK_1$, $G[V_2]=K_{\frac{c(c-1)}{2}}$ and each pair of distinct vertices in $V_1$ has a unique common neighbor in $V_2$. Obviously, $diam(G)=2$. It remains to show that $L_{2,t}(G)=c$. We know $|V(G)|=c+\frac{c(c-1)}{2}$ and $\Delta(G)=\frac{c(c-1)}{2}+1$. Hence, by Theorem \ref{t10}, $L_{2,t}(G)\leq |V(G)|+2-\Delta(G)=c+1$. But $G\notin \Omega$, so $L_{2,t}(G)\leq c$.
On the other hand, $V_1$ is a $2$TLP set of $G$. Thus, $L_{2,t}(G)= c$.
\end{proof}

\begin{theorem}
Assume that $a\geq3$ and $b$ are two integers with $a+1\leq b \leq 2a$. Then, there exists a tree $T$ for which $\rho_o(T)=a$ and $L_{2,t}(T)=b$.
\end{theorem}
\begin{proof}
Let $a\geq3$ and $b$ be two integers such that $a+1\leq b \leq 2a$, and $b=a+x$ with $1\leq x\leq a$. In what follows, we construct a tree $T$ with $\rho_o(T)=a$ and $L_{2,t}(T)=a+x$ for $a\geq3$ and $1\leq x\leq a$.
We distinguish two cases based on the value of $x$.

\textit{Case 1.}
First, let $x=a$. Assume $P=v_1v_2\cdots v_a$ is a path. We add two leaves $u_{i_1}$ and $u_{i_2}$ to each $v_i$, and obtain tree $T$. Let $S_1$ be a $\rho_o(T)$-set. If $|S_1|\geq a+1$, by the Pigeonhole principle, there is at least one vertex $v_i$ such that $|N(v_i)\cap S_1|\geq2$, which is impossible. Hence, $\rho_o(T)\leq a$. On the other hand, $\{u_{1_1}, u_{2_1}, \cdots, u_{a_1}\}$ is a $1$TLP set of $T$, so $\rho_o(T)= a$. 

Let $S_2$ be an $L_{2,t}(T)$-set. Similarly, if $L_{2,t}(T)\geq 2a+1$, there exists at least one vertex $v_i$ such that $|N(v_i)\cap S_2|\geq3$, a contradiction. Thus, $L_{2,t}(T)\leq 2a$. Moreover, $\{u_{1_1}, u_{1_2}, u_{2_1}, u_{2_2} \cdots, u_{a_1}, u_{a_2}\}$ is a $2$TLP set of $T$, hence $L_{2,t}(T)= 2a=b$.

\textit{Case 2.} Suppose now that $1\leq x\leq a-1$. Consider the star $T'=K_{1,a}$ with $V(T')= \{r,v_1,v_2, \cdots, v_a\}$ and deg$(r)=a$. Let $T$ be the tree obtained from $T'$ by adding two leaves $u_i$ and $u'_i$ to each $v_i$ for $1\leq i\leq x-1$ and one leaf $u_i$ to each $v_i$ for $x\leq i\leq a-1$. We show that $\rho_o(T)=a$ and $L_{2,t}(T)=b$.
Since $T\notin \Omega$, it follows that $L_{2,t}(T)<|V(T)|+2-\Delta(T)$ by Theorem \ref{t10}. Notice that $|V(T)|=2a+x-1$ and $\Delta(T)=a$, thus $L_{2,t}(T)\leq a+x$. On the other hand, $\{u_1, u_2, \cdots, u_{a-1}, u'_1, u'_2, \cdots, u'_{x-1}, v_1, v_a\}$ is a $2$TLP set of $T$, so $L_{2,t}(T) = a+x =b$.

Since $T$ is a tree with at least two vertices, $\rho_o(T)=\gamma_t(T)$ \cite{hms}. Moreover, $\{r, v_1, v_2, \cdots, v_{a-1}\}$ is a TD set of $T$, and hence $\gamma_t(T)\leq a$. Thus, $\rho_o(T)\leq a$. It is readily verified that $\{u_1, u_2, \cdots, u_{a-1}, v_a\}$ is a $1$TLP set of $T$. Therefore, $\rho_o(T) =a$.
\end{proof}

\begin{theorem}
Let $G$ have a unique  $L_{2,t}(G)$-set $B$. Then every leaf of $G$ belongs to $B$.
\end{theorem}
\begin{proof}
Let $B$ be a unique  $L_{2,t}(G)$-set, and let there exist a leaf $l\notin B$ with the support vertex $v$. If $v\in B$ and $|N(v)\cap B |\leq1$, then $B'=B\cup \{l\}$ is a $2$TLP set which is greater than $B$, a contradiction. So if $v\in B$, then $|N(v)\cap B |=2$. Let $u\in N(v)\cap B$. We can easily see that   $B''=(B\backslash \{u\})\cup\{l\}$ is an $L_{2,t}(G)$-set, which is impossible because $B$ is unique. Hence $v\notin B$.

If some neighbor of $v$, say $u'$, belongs to $B$, then $B''=(B\backslash \{u'\})\cup \{l\}$ is an $L_{2,t}(G)$-set. This contradicts the assumption. Therefore, we deduce that $N[v]\cap B=\emptyset$. So $B\cup \{l\}$ is a $2$TLP set, which is a contradiction with the maximality of $B$. Hence $l\in B$.
\end{proof}


\section{On $2$-(total) limited packing number of some graph products}

\begin{theorem}\label{Cart1}
For any graphs $G$ and $H$, $L_{2,t}(G\square H)\geq \max\{L_{2,t}(G)\rho(H),\rho(G)L_{2,t}(H)\}$. Moreover, this bound is sharp.
\end{theorem}
\begin{proof}
Let $P_{G}$ and $P_{H}$ be an $L_{2,t}(G)$-set and a $\rho(H)$-set, respectively. Set $P=P_{G}\times P_{H}$, and suppose to the contrary that $P$ is not a $2$TLP set of $G\square H$. Therefore, there exists a vertex $(x,y)\in V(G)\times V(H)$ adjacent to three distinct vertices $(g_{1},h_{1}),(g_{2},h_{2}),(g_{3},h_{3})\in P$. We distinguish the following cases.

\textit{Case 1.} $h_{1}=h_{2}=h_{3}$. If $y$ is adjacent to $h_1$, then $x=g_{1}=g_{2}=g_{3}$, which is impossible. So, $y=h_{1}=h_{2}=h_{3}$. In such a situation, $x$ is adjacent to $g_{1},g_{2},g_{3}\in P_{G}$, which contradicts the fact that $P_{G}$ is a $2$TLP set in $G$.

\textit{Case 2.} At least two vertices from $\{h_{1},h_{2},h_{3}\}$, say $h_{1}$ and $h_{2}$, are distinct. By the adjacency rule of the Cartesian product graphs, we deduce that $\{h_{1},h_{2}\}\subseteq N_{H}[y]\cap P_{H}$. This contradicts the fact that $P_{H}$ is a packing in $H$.

Therefore, $P$ is a $2$TLP set in $G\square H$. Hence, $L_{2,t}(G\square H)\geq|P|=L_{2,t}(G)\rho(H)$. Similarly, we have $L_{2,t}(G\square H)\geq \rho(G)L_{2,t}(H)$.

We can show that this bound is sharp in the following way. Let $G'$ be any connected graph on the set of vertices $\{v_{1}',\cdots,v_{n}'\}$. Let $G=G'\odot K_{1}$, in which $N_{G}(v_{i}')\setminus N_{G'}(v_{i}')=\{v_{i}\}$ for each $1\leq i\leq n$. We now consider the graph $G\square K_{r}$ for $r\geq3$, and let $Q$ be an $L_{2,t}(G\square K_r)$-set. It is not difficult to see that $|Q\cap\big{(}\{v_{i},v_{i}'\}\times V(K_{r})\big{)}|\leq2$ for each $1\leq i\leq n$. This implies that
\begin{equation}\label{BJ1}
\begin{array}{lcl}
L_{2,t}(G)&=&|Q|=|Q\cap V(G\square K_{r})|=|Q\cap\big{(}\cup_{i=1}^{n}(\{v_{i},v_{i}'\}\times V(K_{r}))\big{)}|\\
&=&\sum_{i=1}^{n}|Q\cap\big{(}\{v_{i},v_{i}'\}\times V(K_{r})\big{)}|\leq2n=L_{2,t}(K_{r})\rho(G).
\end{array}
\end{equation}
\end{proof}

\begin{theorem}
Let $G$ and $H$ be graphs with $i_{G}$ and $i_{H}$ isolated vertices, respectively. Then,
$$L_{2,t}(G\times H)\geq max~\{\rho_o(G^{-}) L_{2,t}(H^{-}), L_{2,t}(G^{-}) \rho_o(H^{-})\}+i_G|V(H)|+i_H|V(G)|-i_Gi_H$$ and this bound is sharp.
\end{theorem}
\begin{proof}
Suppose first that $G$ and $H$ are graphs without isolated vertices. Let $P_G$ and $P_H$ be a $\rho_o(G)$-set and an $L_{2,t}(H)$-set, respectively. Set $P=P_G\times P_H$, and assume for the sake of contradiction that $P$ is not a $2$TLP set of $G\times H$. Hence, there exists a vertex $(x,y)\in V(G\times H)$ adjacent to three distinct vertices $(g,h), (g',h'),(g'',h'') \in P$. Then $g=g'=g''$ because $P_G$ is a $\rho_o(G)$-set. So $h\neq h'\neq h''$ and $|N(y)\cap P_H|\geq 3$, a contradiction. Therefore, $P$ is a $2$TLP set in $G\times H$, and $L_{2,t}(G\times H)\geq |P| =\rho_o(G) L_{2,t}(H)$. We have $L_{2,t}(G\times H)\geq L_{2,t}(G) \rho_o(H)$ by a similar fashion.

We have
$$L_{2,t}(G\times H)=L_{2,t}(G^-\times H^-)+i_G|V(H)|+i_H|V(G)|-i_Gi_H.$$
Therefore,
$$L_{2,t}(G\times H)\geq max~\{\rho_o(G^{-}) L_{2,t}(H^{-}), L_{2,t}(G^{-}) \rho_o(H^{-})\}+i_G|V(H)|+i_H|V(G)|-i_Gi_H.$$

In what follows, we show that this bound is sharp. Let $G$ be a bipartite graph without isolated vertices. Then, $L_{2,t}(G\times K_2)=L_{2,t}(2G)=2L_{2,t}(G)$. On the other hand, $L_{2,t}(G\times K_2)\geq max~\{\rho_o(G) L_{2,t}(K_2), L_{2,t}(G) \rho_o(K_2)\}= max~\{2\rho_o(G), 2 L_{2,t}(G)\}=2  L_{2,t}(G)$.
\end{proof}

\begin{theorem}
Let $G$ be a graph of order $n$ with $i_G$ isolated vertices. If $H$ is a rooted graph at $v$, then
$$n(L_{2,t}(H)-1)+i_G \leq L_{2,t}(G\circ_{v} H)\leq n L_{2,t}(H).$$
Furthermore, these bounds are sharp.
\end{theorem}
\begin{proof}
Note that any $H$-layer in $G\circ_{v} H$ is isomorphic to $H$. So, each $L_{2,t}(G\circ_{v} H)$-set intersects every $H$-layer is at most $L_{2,t}(H)$ vertices. Hence, $L_{2,t}(G\circ_{v} H)\leq n L_{2,t}(H)$. On the other side, let $P_H$ be an $L_{2,t}(H)$-set. We can readily observe that $P=\bigcup_{g\in V(G)} \big(\{g\}\times (P_H\backslash\{v\})\big)$ is a $2$TLP set in $G\circ_{v} H$, so we have $n(L_{2,t}(H)-1)\leq|P|\leq L_{2,t}(G\circ_{v} H)$.

Suppose that there exists an $L_{2,t}(H)$-set $P_H$ not consisting of $v$. Notice that $P=\bigcup_{g\in V(G)} \big(\{g\}\times P_H\big)$ is a $2$TLP set in $G\circ_{v} H$, and we have $nL_{2,t}(H)=|P|\leq L_{2,t}(G\circ_{v} H)$. Thus, we conclude that $L_{2,t}(G\circ_{v} H)=nL_{2,t}(H)$ in this case.

Assume now that $v$ has degree two in all subgraphs induced by every $L_{2,t}(H)$-set $P_H$, that is $deg_{H[P_H]}(v)=2$.
Suppose that for every $2$TLP set $P'_H$ in $H$, $|P'_H\backslash N[v]|\leq L_{2,t}(H)-3$.
Assume that $P$ is an $L_{2,t}(G\circ_{v} H)$-set. Note that exactly $i_G$ components of $G\circ_{v} H$ are isomorphic to $H$, which implies that every component has exactly $L_{2,t}(H)$ vertices in $P$. Moreover, we have one component isomorphic to $G^-\circ_{v} H$.
Let $P^-=P\cap (G^-\circ_{v} H)$ and $P^-_g=P^-\cap {}^g\!H$ for every $g \in V(G^-)$. We now show that $L_{2,t}(G^-\circ_{v} H)=(n-i_G) (L_{2,t}(H)-1)$. Assume that $L_{2,t}(G^-\circ_{v} H)>(n-i_G) (L_{2,t}(H)-1)$. Hence, there exists at least one vertex $g_1 \in V(G^-)$ for which $|P^-_{g_1}|=L_{2,t}(H)$. Otherwise, $|P^-|= \sum_{g_i\in V(G^-)}|P^-_{g_i}|\leq \sum_{g_i\in V(G^-)} L(_{2,t}(H)-1)=(n-i_G) (L_{2,t}(H)-1)$, which is a contradiction. 

Since $G^-$ has no isolated vertex, there exists a vertex $g_2\in V(G^-)$ for which $g_1g_2\in E(G)$. If $|P^-_{g_2}|=L_{2,t}(H)$, then both $(g_1,v)$ and $(g_2,v)$ have three neighbors in $P^-$, which is impossible. Therefore $|P^-_{g_2}|\leq L_{2,t}(H)-1$ for all $g_2\in N_{G}(g_1)$. Now let $g_2$ be an arbitrary neighbor of $g_1$ in $G^-$. If $|P^-_{g_2}|=L_{2,t}(H)-1$, then $|N_{(G\circ_{v} H)[{}^{g_2}\!H]}[(g_2,v)]\cap P^-_{g_2}|=2$. This implies that $|N_{(G^-\circ_{v} H)}(g_1,v)\cap P^-|\geq 3$ or $|N_{(G^-\circ_{v} H)}(g_2,v)\cap P^-|\geq 3$, a contradiction. Thus, $|P^-_{g_2}|\leq L_{2,t}(H)-2$ for all $g_2\in N_{G}(g_1)$.

The above argument provides a guarantee that for every vertex $g_1\in V(G)$ such that $|P^-_{g_1}|=L_{2,t}(H)$, we have $|P^-_{g_2}|\leq L_{2,t}(H)-2$ for all $g_2\in N_{G}(g_1)$. This implies that $$L_{2,t}(G^-\circ_{v} H)\leq \frac{n-i_G}{2}L_{2,t}(H)+\frac{n-i_G}{2}(L_{2,t}(H)-2)=(n-i_G)(L_{2,t}(H)-1),$$ which contradicts the assumption $L_{2,t}(G^-\circ_{v} H)>(n-i_G)(L_{2,t}(H)-1)$. So $L_{2,t}(G^-\circ_{v} H)\leq(n-i_G)(L_{2,t}(H)-1)$. It follows that $L_{2,t}(G^-\circ_{v} H)=(n-i_G)(L_{2,t}(H)-1)$ by using the corresponding inequality obtained from the first steps of the proof. Therefore, $L_{2,t}(G\circ_{v} H)=L_{2,t}(G^-\circ_{v} H)+i_GL_{2,t}(H)=(n-i_G)(L_{2,t}(H)-1)+i_GL_{2,t}(H)=n(L_{2,t}(H)-1)+i_G$.
\end{proof}

\begin{theorem}
For any graphs $G$ and $H$, $$L_2(G\square H)\leq \min\{L_2(G)|V(H)|,L_2(H)|V(G)|\},$$ and this bound is sharp.
\end{theorem}
\begin{proof}
Let $V(H)=\{v_1, v_2, \cdots, v_{|V(H)|}\}$. It is obvious that $G\square H$ contains $|V(H)|$ disjoint $G$-layers. Suppose now that $P$ is an $L_2(G\square H)$-set, thus $P_i=P\cap G^{v_i}$ is a $2$LP set in $(G\square H)[G^{v_i}]$ for each $1\leq i\leq |V(H)|$. Therefore, $|P_i|\leq L_2(G)$, which leads to $$L_2(G\square H)=|P|=\sum_{i=1}^{|V(H)|} |P_i|\leq L_2(G)|V(H)|.$$
Similarly, we have $L_2(G\square H)\leq L_2(H)|V(G)|$.

For sharpness consider $G=P_2$ and $H=K_{m,n}$ for $m,n\geq2$. We observe that $L_2(K_{m,n})=2$ \cite{gghr}, and $L_2(P_2)=2$. It is easy to see that $L_2(G\square H)=4$.
\end{proof}

\begin{theorem}
Let $G$ and $H$ be graphs with $i_{G}$ and $i_{H}$ isolated vertices, respectively. Then,
\begin{center}
$L_2(G\times H)\geq max~\{\rho_o(G^{-})L_2(H^{-}), L_2(G^{-}) \rho_o(H^{-}), \rho(G^{-}) L_{2,t}(H^{-}), L_{2,t}(G^{-}) \rho(H^{-}) \}$\\
$+i_G|V(H)|+i_H|V(G)|-i_Gi_H$.
\end{center}
Moreover, this bound is sharp.
\end{theorem}
\begin{proof}
Assume first that $G$ and $H$ are graphs without isolated vertices. Let $P_G$, $P'_G$, $P_H$ and $P'_H$ be an $L_{2,t}(G)$-set, a $\rho_o(G)$-set, a $\rho(H)$-set and an $L_2(H)$-set, respectively. Set $P=P_G\times P_H$ and $P'=P'_G\times P'_H$, and suppose to the contrary that $P$ and $P'$ are not $2$LP sets of $G\times H$. So there exist vertices $(x,y)$, $(x',y')\in V(G\times H)$ such that $|N[(x,y)]\cap P|\geq3$ and $|N[(x',y')]\cap P'|\geq3$, respectively.

If $(x,y)\in P$, then $(x,y)$ is adjacent to two distinct vertices $(g,h), (g',h')\in P$. So $|N[y]\cap P_H|\geq2$, which is impossible.
If $(x,y)\in V(G\times H)\backslash P$, then $(x,y)$ is adjacent to three distinct vertices $(g,h), (g',h'), (g'',h'') \in P$. We observe that $h=h'=h''$ because $P_H$ is a $\rho(H)$-set. Hence $g\neq g' \neq g''$ and $|N(x)\cap P_G|\geq3$, a contradiction.
Therefore, $P$ is a $2$LP set in $G\times H$ and $L_2(G\times H)\geq |P|\geq L_{2,t}(G) \rho(H)$. We have $L_2(G\times H)\geq  \rho(G)L_{2,t}(H)$ by a similar method.

If $(x',y')\in P'$, then $(x',y')$ is adjacent to two distinct vertices $(g_1,h_1), (g_2,h_2)\in P'$. We have $g_1=g_2$ because $P'_G$ is a $\rho_o(G)$-set. Thus, $h_1\neq h_2$ and $|N[y']\cap P'_H|\geq3$, which is impossible.
If $(x',y')\in V(G\times H)\backslash P'$, then $(x',y')$ is adjacent to three distinct vertices $(g_1,h_1), (g_2,h_2), (g_3,h_3) \in P'$. $g_1=g_2=g_3$ since $P'_G$ is a $\rho_o(G)$-set, so $h_1\neq h_2 \neq h_3$ and $|N[y']\cap P'_H|\geq3$, a contradiction.
Therefore $P'$ is a $2$LP set in $G\times H$, and $L_2(G\times H)\geq \rho_o(G) L_2(H)$. We get $L_2(G\times H)\geq  L_2(G)\rho_o(H)$ by a similar fashion.
Therefore, $$L_2(G\times H)\geq max~\{\rho_o(G)L_2(H), L_2(G) \rho_o(H), \rho(G) L_{2,t}(H), L_{2,t}(G) \rho(H)\}.$$

We now suppose that $G$ and $H$ are arbitrary graphs. Then,
$$L_2(G\times H)=L_2(G^-\times H^-)+i_G|V(H)|+i_H|V(G)|-i_Gi_H \geq$$
$$max~\{\rho_o(G^{-})L_2(H^{-}), L_2(G^{-}) \rho_o(H^{-}), \rho(G^{-}) L_{2,t}(H^{-}), L_{2,t}(G^{-}) \rho(H^{-}) \}+$$
$$i_G|V(H)|+i_H|V(G)|-i_Gi_H.$$

In what follows, we show the sharpness of this bound. Let $G$ be a bipartite graph without isolated vertices. Then, $L_2(G\times K_2)=L_2(2G)=2L_2(G)$. On the other hand,
$$L_2(G\times K_2)\geq max~\{\rho_o(G) L_2(K_2), L_2(G) \rho_o(K_2), \rho(G) L_{2,t}(K_2), L_{2,t}(G) \rho(K_2)\}=$$
$$max~\{2\rho_o(G), 2 L_2(G), 2 \rho (G), L_{2,t}(G)\}=2  L_2(G).$$
\end{proof}

We end this section by studying the $2$LP number of rooted product graphs.
\begin{theorem}
Let $G$ be a graph of order $n$. If $H$ is a graph with root $v$, then\\
$$
L_2(G\circ_{v} H)=
\begin{cases}
L_2(G)+n(L_2(H)-1) \quad \text{if}~v \in P_H~\text{for every}~L_2(H)\text{-set}~P_H,\\
nL_2(H) \qquad ~~~~~~~~~~~~~~~\text{if}~v \notin P_H~\text{for some}~L_2(H)\text{-set}~P_H.
\end{cases}
$$
\end{theorem}
\begin{proof}
We consider two cases based  on the membership of $v$ to $L_2(H)$-sets.

\textit{Case 1.} Assume that $v$ belongs to any $L_2(H)$-set $P_H$, and $P'$ be an $L_2(G)$-set. Set $P=(P'\times \{v\})\cup \big(V(G)\times (P_H\backslash\{v\})\big)$. It can be readily seen that $P$ is a $2$LP set in $G\circ_{v} H$. Therefore, $L_2(G\circ_{v} H)\geq |P'\times \{v\}|+|V(G)\times (P_H\backslash\{v\})|=L_2(G)+n(L_2(H)-1)$.

On the other hand, let $B$ be an $L_2(G\circ_{v} H)$-set. Then $B_g=B\cap {}^g\!H$ is a $2$LP set in $(G\circ_{v} H)[{}^g\!H]$ for every $g \in V(G)$. Note that $B_g$ is not an $L_2((G\circ_{v} H)[{}^g\!H])$-set for some $g\in V(G)$ since $v$ belongs to every $L_2(H)$-set. Hence, $|B\cap {}^g\!H|=|B_g|\leq L_2(H)-1$ if $(g,v)\notin B$, which means $|B\cap ({}^g\!H|\backslash \{(g,v)\})|\leq L_2(H)-1$.
Also if $(g,v)\in B$, then $|B\cap ({}^g\!H|\backslash \{(g,v)\})|\leq L_2(H)-1$ as well. In addition, $B\cap G^v$ is a $2$LP set in $(G\circ_{v} H)[G^v]$. Thus, $|B\cap G^v|\leq L_2(G)$, and we have
$$L_2(G\circ_{v} H)=|B|=|B\cap G^v|+\sum_{g\in V(G)}|B\cap({}^g\!H|\backslash \{(g,v)\})|\leq L_2(G)+n(L_2(H)-1).$$
Therefore, $L_2(G\circ_{v} H)=L_2(G)+n(L_2(H)-1)$.

\textit{Case 2.} Assume that there exists an $L_2(H)$-set $P_H$ for which $v\notin P_H$. Let ${}^g\!P_H=\{g\}\times P_H$ for every $g\in V(G)$, and let $P''=\cup_{g\in V(G)}{}^g\!P_H$. We can easily see that $P''$ is a $2$LP set in $G\circ_{v} H$, so $L_2(G\circ_{v} H)\geq|P''|=nL_2(H)$. On the other hand, let $P$ be an $L_2(G\circ_{v} H)$-set. We can easily observe that the set $P_g=P\cap {}^g\!H$ is a $2$LP set in $(G\circ_{v} H)[{}^g\!H]$ for every $g\in V(G)$. So $L_2(H)=L_2(G\circ_{v} H)[{}^g\!H]\geq |P_g|$. Therefore, $L_2(G\circ_{v} H)=|P|=\sum_{g\in V(G)}|P_g|\leq \sum_{g\in V(G)}L_2(H)=nL_2(H)$. This leads to $L_2(G\circ_{v} H)=nL_2(H)$. 
\end{proof}

\section{Results on vertex partitioning into $2$-limited packings}

\begin{theorem}
If $G$ is a graph of order $n\geq 2$, then $\chi_{\times k}(G)\geq2\sqrt n-L_k(G)$.
\end{theorem}
\begin{proof}
We first prove that $\chi_{\times k}(G)\times L_k(G)\geq n$. Let $\{B_1, B_2,\cdots, B_{\chi_{\times k}(G)}\}$ be a $k$LPP of $G$. Then,
$$\chi_{\times k}(G)\times L_k(G)=\sum_{i=1}^{\chi_{\times k}(G)} L_k(G)\geq\sum_{i=1}^{\chi_{\times k}(G)} |B_i|=n$$
and equality holds when each set $B_i$ is an $L_k(G)$-set. So $\chi_{\times k}(G)+L_k(G)\geq \chi_{\times k}(G)+\frac{n}{\chi_{\times k}(G)}$.

On the other hand, $\chi_{\times k}(G)\leq \frac {n}{k}$ because every subset of $V(G)$ of cardinality at most $k$ is a $k$LP set.
We observe that the function $g(x)=x+\frac {n}{x}$ is decreasing for $1\leq x\leq \sqrt n$, and it is increasing for $\sqrt n \leq x \leq \frac {n}{k}$. Therefore $\chi_{\times k}(G)+L_k(G)\geq 2\sqrt n$.

This bound is sharp for the complete graph $K_4$, the cycle $C_4$ and the star $S_4$.
\end{proof}

The \textit{corona product} of two graphs $G$ with $V(G)={v_1, . . . , v_n}$ and $H$ is defined as the graph created by taking one copy of $G$, $|V(G)|$ copies of $H$ and joining $v_i \in V(G)$ to any vertex in the $i^{th}$ copy of $H$. The corona product of the graphs $G$ and $H$ is denoted by $G\odot H$.

If $v$ is a vertex of maximum degree in $G$, then $N_{G\odot H}[v]=\Delta(G)+1+|V(H)|$. So we need at least $\lceil \frac{\Delta(G)+1+|V(H)|}{2}\rceil$ $2$-limited packing sets in every 2LPP of $G\odot H$. Therefore $\chi_{\times 2}(G\odot H)\geq \lceil \frac{\Delta(G)+1+|V(H)|}{2}\rceil$.

\begin{theorem}
If $G$ and $H$ are two graphs, then
$$\chi_{\times 2}(G\odot H) \in \{\chi_{\times 2}(G), \chi_{\times 2}(G)+1, \chi_{\times 2}(G)+2, \cdots, \chi_{\times 2}(G)+\lceil \frac{|V(H)|}{2}\rceil\}.$$
\end{theorem}
\begin{proof}  
Let $V(G)=\{v_1, v_2, \cdots, v_n\}$ and $V(H)=\{u_1, u_2, \cdots, u_{n'}\}$, and let $\mathbb{P}=\{P_1, P_2, \cdots, P_{\chi_{\times 2}(G)}\}$ be a $2$LPP of $G$.

Since $G$ is a subgraph of $G\odot H$, $\chi_{\times 2}(G)\leq \chi_{\times 2}(G\odot H)$. That $\chi_{\times 2}(G\odot H)=\chi_{\times 2}(G)$ can be seen as follows. If $|N_G[v_i]\cap(\cup_{j=1}^{\chi_{\times 2}(G)} P_j)|\leq 2\chi_{\times 2}(G)-n'$ for each $v_i\in V(G)$, then we place the vertices of each copy of $H$ in the members of $\mathbb{P}$ such that $|N_{G\odot H}[v_i]\cap P_j)|\leq 2$ for each $1\leq i \leq n$ and $1\leq j \leq \chi_{\times 2}(G)$. So the equality holds.

We now have two cases based on the behavior of $|V(H)|$ and prove the upper bound.\\
$\bullet$ Let $|V(H)|$ be even. In the worst case, if there exists a vertex $v_i\in V(G)$ such that $2\chi_{\times 2}(G)-1\leq |N_G[v_i]\cap(\cup_{j=1}^{\chi_{\times 2}(G)} P_j)|\leq 2\chi_{\times 2}(G)$, then we add new sets $P_{\chi_{\times 2}(G)+1}, P_{\chi_{\times 2}(G)+2},\cdots, P_{\chi_{\times 2}(G)+\lceil \frac{|V(H)|}{2}\rceil}$ to $\mathbb{P}$ and put the vertices of each copy of $H$ in these sets two by two until there are no vertices left. Therefore $\chi_{\times 2}(G\odot H)=\chi_{\times 2}(G)+\lceil \frac{|V(H)|}{2}\rceil$.\\
$\bullet$ Let $|V(H)|$ be odd. If there exists a vertex $v_i\in V(G)$ such that $|N_G[v_i]\cap(\cup_{j=1}^{\chi_{\times 2}(G)} P_j)|= 2\chi_{\times 2}(G)$, then $\chi_{\times 2}(G\odot H)=\chi_{\times 2}(G)+\lceil \frac{|V(H)|}{2}\rceil$ as before.
Hence $\chi_{\times 2}(G)\leq \chi_{\times 2}(G\odot H)\leq \chi_{\times 2}(G)+\lceil \frac{|V(H)|}{2}\rceil$.

In what follows, we show that $\chi_{\times 2}(G\odot H)$ can take all values between $\chi_{\times 2}(G)$ and $\chi_{\times 2}(G)+\lceil \frac{|V(H)|}{2}\rceil$. It is enough to consider the graph $G$ as $K_{a,a}$ and graph $H$ as $\overline {K_{a+b-1}}$ for $a\geq 1$ and $b\geq 0$. Let $\mathbb{P}=\{P_1, P_2, \cdots, P_{\chi_{\times 2}(G)}\}$ be a $2$LPP of $G$. If $v_i\in V(G)$ belongs to $P_j$, then $v_i$ has the label $j$. We first show that $\chi_{\times 2}(G\odot H)=\chi_{\times 2}(G)+\lceil \frac{b}{2}\rceil$. We assign the labels from $\{1,2, \cdots, a\}$ to the vertices of $G$ as shown in Figure \ref{H}, which is equivalent to a $2$LPP for $G$ with the smallest cardinality.

\begin{figure}[h]
\centering
\begin{tikzpicture}[scale=.5, transform shape]
\node [draw, shape=circle] (1) at (0,0) {};
\node [draw, shape=circle] (2) at (2,0) {};
\node [draw, shape=circle] (3) at (4,0) {};
\node [draw, shape=circle] (4) at (6,0) {};
\node [draw, shape=circle] (a) at (10,0) {};
\node [draw, shape=circle] (1') at (0,-2) {};
\node [draw, shape=circle] (2') at (2,-2) {};
\node [draw, shape=circle] (3') at (4,-2) {};
\node [draw, shape=circle] (4') at (6,-2) {};
\node [draw, shape=circle] (a') at (10,-2) {};

\node [scale=1.5] at (0,0.5) {$1$};
\node [scale=1.5] at (2,0.5) {$2$};
\node [scale=1.5] at (4,0.5) {$3$};
\node [scale=1.5] at (6,0.5) {$4$};
\node [scale=1.5] at (10,0.5) {$a$};
\node [scale=1.5] at (0,-2.5) {$1$};
\node [scale=1.5] at (2,-2.5) {$2$};
\node [scale=1.5] at (4,-2.5) {$3$};
\node [scale=1.5] at (6,-2.5) {$4$};
\node [scale=1.5] at (10,-2.5) {$a$};

\draw(1)--(1');
\draw(1)--(2');
\draw(1)--(3');
\draw(1)--(4');
\draw(1)--(a');
\draw(2)--(1');
\draw(2)--(2');
\draw(2)--(3');
\draw(2)--(4');
\draw(2)--(a');
\draw(3)--(1');
\draw(3)--(2');
\draw(3)--(3');
\draw(3)--(4');
\draw(3)--(a');
\draw(4)--(1');
\draw(4)--(2');
\draw(4)--(3');
\draw(4)--(4');
\draw(4)--(a');
\draw(a)--(1');
\draw(a)--(2');
\draw(a)--(3');
\draw(a)--(4');
\draw(a)--(a');

\end{tikzpicture}
\caption{the labeling of $V(G)$.}\label{H}
\end{figure}
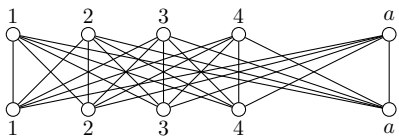 

We have $N_{G\odot H}[v_i]=N_G(v_i)+1+|V(H)|=2a+b$ for each $v_i\in V(G)$. $(a+1)$ closed neighbors of each $v_i\in V(G)$ are labeled as above. If $v_i\in P_k$, i.e., $v_i$ has label $k$, then $|N_G[v_i]\cap P_k|=2$ and $|N_G[v_i]\cap P_j|=1$ for each $j\neq k$, as we see in figure \ref{H}. It is clear that $j$ has $a-1$ values. We put $(a-1)$ unlabeled neighbors of $v_i$ one by one in the sets $P_j$ with the previous condition. Thus, $b$ vertices are still unlabeled. We consider the new labels $a+1, a+2, \cdots, a+\lceil \frac{b}{2}\rceil$, and then label the remaining $b$ vertices two by two with them. Therefore, $\chi_{\times 2}(G\odot H)=\chi_{\times 2}(G)+\lceil \frac{b}{2}\rceil$.

 Note that if $a=1$ and $b\geq 1$, then $\chi_{\times 2}(G\odot H)=\chi_{\times 2}(G)+\lceil \frac{b+1-1}{2}\rceil=\chi_{\times 2}(G)+\lceil \frac{|V(H)|}{2}\rceil$. If $b=0$, then $\chi_{\times 2}(G\odot H)=\chi_{\times 2}(G)$.
\end{proof}


\section{Concluding remarks}

We have bounded $\chi_{\times2}$ for some product graphs in our paper. It may be of interest to provide the exact value of this parameter and either improve or bound it for other product graphs.\vspace{4.5mm}


\end{document}